\documentclass [12pt] {article}
\usepackage{amsthm,amsmath}
\usepackage{epsfig}
\usepackage{amssymb}
\usepackage[english]{babel}

\newtheorem {theorem} {Theorem}
\newtheorem {lemma} {Lemma}

\newtheorem {cor} {cor}

\newcommand {\ttbox} [1] {\mbox {\ttfamily*1}} 

\title{Optimal recovery of operator sequences}
\author{V.\,F.~Babenko, N.\,V.~Parfinovych, D.\,S.~Skorokhodov}
\date{\today}

\begin{document}

\maketitle

\begin{abstract}
    In this paper we consider two recovery problems based on information given with an error. First is the problem of optimal recovery of the class $W^T_q = \{(t_1h_1,t_2h_2,\ldots)\in \ell_q\,:\,\|h\|_q\leqslant 1\}$, where $1\le q < \infty$ and $t_1\geqslant t_2\geqslant \ldots \geqslant 0$, in the space $\ell_q$ when in the capacity of inexact information we know either the first $n\in\mathbb{N}$ elements of a sequence with an error measured in the space of finite sequences $\ell_r^n$, $0 < r \le \infty$, or a sequence itself is known with an error measured in the space $\ell_r$. The second is the problem of optimal recovery of scalar products acting on Cartesian product $W^{T,S}_{p,q}$ of classes $W^T_p$ and $W^S_q$, where $1 < p,q < \infty$, $\frac{1}{p} + \frac{1}{q} = 1$ and $s_1\ge s_2\ge \ldots \ge 0$, when in the capacity of inexact information we know the first $n$ coordinate-wise products $x_1y_1, x_2y_2,\ldots,x_ny_m$ of the element $x\times y \in W^{T,S}_{p,q}$ with an error measured in the space $\ell_r^n$. We find exact solutions to above problems and construct optimal methods of recovery. As an application of our results we consider the problem of optimal recovery of classes in Hilbert spaces by Fourier coefficients known with an error measured in the space $\ell_p$ with $p > 2$.
\end{abstract}

\section{Introduction}
\label{Intro}

Let $X,Z$ be complex linear spaces, $Y$ be a complex normed space, $A:X\to Y$ be an operator, in general non-linear, with domain $\mathcal{D}(A)$, $W\subset \mathcal{D}(A)$ be some class of elements. Denote by $\mathfrak{B}(Z)$ the set of non-empty subsets of $Z$, and let $I: \overline{{\rm span}\,W}\to \mathfrak{B}(Z)$ be a given mapping called {\it information}. When saying that information about element $x\in W$ is available we mean that some element $z\in I(x)$ is known. An arbitrary mapping $\Phi: Z\to Y$ is called {\it method of recovery} of operator $A$. Define {\it the error of method of recovery} $\Phi$ of operator $A$ on the set $W$ given information $I$:
\begin{equation}
\label{recovery_error}
    \mathcal{E}(A,W,I,\Phi) = \sup_{x\in W} \sup_{z\in I(x)}\left\|Ax-\Phi(z)\right\|_Y.
\end{equation}
The quantity
\begin{equation}
\label{optimal_recovery}
    \mathcal{E}(A,W,I) = \inf_{\Phi:Z\to Y} \mathcal{E}(A,W,I,\Phi)
\end{equation}
is called {\it the error of optimal recovery} of operator $A$ on elements of class $W$ given information $I$. Method $\Phi^*$ delivering $\inf$ in~\eqref{optimal_recovery} (if any exists) is called {\it optimal}.

The problem of recovery of linear operators in Hilbert spaces based on exact information was studied in~\cite{MicRiv_77}. In case information mapping $I$ has the form $Ix = i(x) + B$, where $i$ is a linear operator and $B$ is a ball of some radius defining information error, recovery problem~\eqref{optimal_recovery} was considered in~\cite{MelMic_79} (see also~\cite{MicRiv_84}--\cite{Pla_96}). Alternative approach to the study of optimal recovery problems based on standard principles of convex optimization was proposed in~\cite{MagOsi_02}. In~\cite{MelMic_79} it was shown that among optimal methods of recovery there exists a linear one, and in~\cite{MagOsi_02} explicit representations for optimal methods of recovery were found in cases when the error of information is measured with respect to the uniform metric. For a thorough overview of optimal recovery and related problems we refer the reader to books~\cite{TraWozWas_88,Pla_96} and survey~\cite{Are_96}.

Remark that results of the present work supplement and generalize results of paper~\cite{MagOsi_02} on optimal recovery of functions and its derivatives and paper~\cite{BabGunPar_20}.

\section{Elementary lower estimate}

Let us present a trivial yet effective lower estimate for the error of optimal recovery~\eqref{optimal_recovery}. Denote by $\theta_Z$ the null element of space $Z$ and let $I$ be some information mapping.

\begin{lemma}
\label{lem_lower_estimate}
    {\sl Let $\theta_Z\in I(W)$. Then
    \[
        \mathcal{E}(A,W,I)\ge \frac 12 \sup\limits_{x,y\in W:\atop \theta_Z\in Ix\cap Iy} \|Ax - Ay\|_Y.
    \]}
\end{lemma}

\begin{proof}[Proof]
    Indeed, for every method of recovery $\Phi:Z\to Y$,
    \[
        \begin{array}{rcl}
            \mathcal{E}(A,W,I,\Phi) & \ge & \sup\limits_{x\in W:\atop \theta_Z\in Ix}\left\|Ax-\Phi\left(\theta_Z\right)\right\|_Y \\ 
            &\ge & \displaystyle \frac{1}{2}\left(\sup\limits_{x\in W:\atop \theta_Z\in Ix}\left\|Ax-\Phi\left(\theta_Z\right)\right\|_Y + \sup\limits_{y\in W:\atop \theta_Z\in Iy}\left\|Ay-\Phi\left(\theta_Z\right)\right\|_Y\right) \\
            & \ge & \displaystyle \frac 12\sup\limits_{x,y\in W:\atop \theta_Z\in Ix\cap Iy} \|Ax - Ay\|_Y.
        \end{array}
    \]
    Taking $\inf$ over methods $\Phi$ we finish the proof. \qedhere
\end{proof}

From Lemma~\ref{lem_lower_estimate} we easily derive the following consequences.

\begin{cor}
\label{cor_identity_operator}
    {\sl Let $A$ be an odd operator, $\tilde{x}\in W$ be such that $-\tilde{x}\in W$ and $\theta_Z\in I (\tilde{x})\cap I (-\tilde{x}).$ Then 
    \[
        \mathcal{E}(A,W,I)\ge\|A\tilde{x}\|_X.
    \]}
\end{cor}

\begin{cor}
\label{cor_scalar_product}
    {\sl Let $Y = \mathbb{C}$, $R$ be a (complex) normed space, $X = R \times R^*$, $W_1\subset R$ and $W_2\subset R^*$ be given classes. Also, let $A$ be the scalar product of elements in $R\times R^*$, i.e. $A(x,y) = \left<y, x\right>$, $x\in R$ and $y\in R^*$. Assume that there exist $\tilde{x}_1\in W_1$ and $\tilde{x}_2\in W_2$ such that either
    \[
        -\tilde{x}_1\in W_1\quad\text{and}\quad \theta_Z\in I\left(\tilde{x}_1,\tilde{x}_2\right)\cap I\left(-\tilde{x}_1,\tilde{x}_2\right)
    \]
    or
    \[
        -\tilde{x}_2\in W_2\quad\text{and}\quad \theta_Z\in I\left(\tilde{x}_1,\tilde{x}_2\right)\cap I\left(\tilde{x}_1,-\tilde{x}_2\right).
    \]
    Then 
    \[
		\mathcal{E}\left(A,W_1\times W_2, I\right)\ge\left|\left< \tilde{x}_2,\tilde{x}_1\right>\right|.
    \]}
\end{cor}

Remark that similar and related lower estimates were established in many papers (see, {\it e.g.},~\cite{MagOsi_02,BabBabParSko_16}).

\section{Optimal recovery of sequences}
\label{sec_3}
    
    Let us present notations used in the rest of the paper. Let $1\le p,q\le \infty$, $\ell_q$ be the standard space of sequences $x=\{ x_k\}_{k=1}^\infty$, complex-valued in general, with corresponding norm $\|x\|_q$, and $\ell_q^n$, $n\in\mathbb{N}$, be the spaces of finite sequences. Denote by $\theta$ the null element of $\ell_q$ and by $\theta^n$ the null element of $\ell_q^n$. 
    
    For a given non-increasing sequence $t=\{ t_k\}_{k=1}^\infty$ of non-negative numbers, consider bounded operator $T:\ell_q\to \ell_q$ defined as follows
    \[
        Th:=\{t_kh_k\}_{k=1}^\infty,\qquad h\in \ell_q,
    \]
    and the class
    \[
        W^T_{q}:=\{ x=Th\; : \; h\in \ell_q,\; \| h\|_q\le 1\}.
    \]
    
    In this section we will study the problem of optimal recovery of identity operator $A = \textrm{id}_X$ on the class $W^T_q$, also called the problem of optimal recovery of class $W^T_{q}$, when information mapping $I$ is given in one of the following forms:
    \begin{enumerate}
        \item $Ix = I_{\overline{\varepsilon}}^nx = \left(x_1,\ldots,x_n\right) + B\left[\varepsilon_1\right]\times B\left[\varepsilon_n\right]$, where $n\in\mathbb{N}$, $\varepsilon_1,\ldots,\varepsilon_n \ge 0$ and $B[\varepsilon_j] = \left[-\varepsilon_j,\varepsilon_j\right]$;
        \item $Ix = I_{\varepsilon,p}^nx = \left(x_1,\ldots,x_n\right) + B\left[\varepsilon,\ell_p^n\right]$, where $n\in\mathbb{N}$, $\varepsilon \ge 0$ and $B\left[\varepsilon,\ell_p^n\right]$ is the ball of radius $\varepsilon$ in the space $\ell_p^n$ centered at $\theta^n$;
        \item $Ix = I_{\varepsilon,p}x = x + B\left[\varepsilon,\ell_p\right]$, where $\varepsilon \ge 0$ and $B\left[\varepsilon,\ell_p\right]$ is the ball of radius $\varepsilon$ in the space $\ell_p$ centered at $\theta$.
    \end{enumerate}
    
    To simplify further notations, we set
    \[
        \mathcal{E}(W,I) := \mathcal{E}(\textrm{id}_X, W,I),\qquad \mathcal{E}(W,I,\Phi) := \mathcal{E}(\textrm{id}_X, W,I,\Phi),
    \]
    and, for $m\in\mathbb{N}$ and $q < \infty$, introduce the method of recovery $\Phi_m^*:\ell_p\to \ell_q$:
    \[
        \Phi_m^*(a)=\left\{a_1\bigg(1-\frac{t^q_{m+1}}{t^q_{1}}\bigg),\dots,a_m\bigg(1-\frac{t^q_{m+1}}{t^q_{m}}\bigg),0,\ldots\right\},\qquad a\in \ell_p,
    \]
    that would be optimal in many situations. Also, we set $\Phi_0^*(a) := \theta$, $a\in\ell_p$.
    
    
    
    In what follows we define $\sum_{k=1}^0 a_k := 0$ for numeric $a_k$'s. 
    In addition, for simplicity we assume that $t_k > 0$ for every $k\in\mathbb{N}$. Results in this paper remain true in the case when $t_k$ can attain zero value with the substitution of ${1}/{t_k}$ with $+\infty$ and ${t_{s}}/{t_k}$, $s\ge k$ with $1$.
    
\subsection{Information mapping $I_{\bar{\varepsilon}}^n(x) = (x_1,\ldots,x_n) + B[\varepsilon_1]\times\ldots\times B[\varepsilon_n]$}

\begin{theorem}
\label{thm_1}
    {\sl Let $n\in \mathbb{N}$, $1\le q < \infty$ and $\varepsilon_1,\ldots,\varepsilon_n\ge 0$. If 
    \[
        1-\sum\limits_{k=1}^{n}\frac{\varepsilon_k^q}{t_k^q}\ge 0,
    \]
    we set $m=n$. Otherwise we choose $m\in \mathbb{Z}_+$, $m\le n$, to be such that
    \[
        1-\sum\limits_{k=1}^m\frac{\varepsilon_k^q}{t_k^q}\ge 0 \qquad \text{and}\qquad 1-\sum\limits_{k=1}^{m+1}\frac{\varepsilon_k^q}{t_k^q}<0.
    \]
    Then 
    \[
        \mathcal{E}\left(W^T_{q},I^n_{\bar \varepsilon}\right) = \mathcal{E}\left(W^T_{q},I^n_{\bar\varepsilon},\Phi_m^*\right) = \left(t_{m+1}^q+\sum\limits_{k=1}^m \left(1- \frac{t_{m+1}^q}{t_k^q}\right)\varepsilon_k^q\right)^{1/q}.
    \]
    }
\end{theorem}

\begin{proof}[Proof]
    Using convexity inequality, relations $\left|x_k-a_k\right|\le \varepsilon_k$, $k=1,\ldots,n$, and monotony of the sequence $t$, we obtain that, for $x = Th\in W^T_{q}$ and $a\in I_{\bar\varepsilon}^n(x)$,
    \begin{gather*}
        \left\|x-\Phi^\ast_m(a)\right\|_q^q=\sum_{k=1}^m\left|x_k-a_k\left(1-\frac{t_{m+1}^q}{t_k^q}\right)\right|^q+\sum_{k=m+1}^\infty |x_k|^q \\
        =\sum_{k=1}^m\left|\bigg(1-\frac{t_{m+1}^q}{t_k^q}\bigg)(x_k-a_k)+\frac{t_{m+1}^q}{t_k^q}x_k\right|^q+\sum_{k=m+1}^\infty t_k^q|h_k|^q \\
        \le\sum_{k=1}^m\left(\bigg(1-\frac{t_{m+1}^q}{t_k^q}\bigg)|x_k-a_k|^q+\frac{t_{m+1}^q}{t_k^q}|x_k|^q\right)+t_{m+1}^q\sum_{k=m+1}^\infty |h_k|^q \\
        =\sum_{k=1}^m \bigg( 1-\frac{t_{m+1}^q}{t_k^q}\bigg)|x_k-a_k|^q+\sum_{k=1}^m t_{m+1}^q|h_k|^q+t_{m+1}^q\sum_{k=m+1}^\infty |h_k|^q \\
        \le\sum_{k=1}^m \bigg( 1-\frac{t_{m+1}^q}{t_k^q}\bigg)\varepsilon_k^q+t_{m+1}^q.
    \end{gather*}

    To obtain the lower estimate, we choose
    \[
        u_k := \frac{\varepsilon_k}{t_k},\quad k=1,\ldots,m,\qquad\text{and}\qquad
        u_{m+1} := \left(1-\sum\limits_{k=1}^m\frac{\varepsilon_{k}^q}{t_k^q}\right)^{1/q},
    \]
    and consider $h^* = \left(u_1,\ldots,u_{m+1},\ldots\right)\in l_q$. It is clear that $Th^*\in W^T_q$, as $\left\|h^\ast\right\|_q \le 1$. Furthermore, by the choice of number $m$ we have that $\theta\in I_{\bar\varepsilon}^n\left(Th^\ast\right)$. Hence, by Corollary~\ref{cor_identity_operator}, 
    \begin{gather*}
        \left(\mathcal{E}\left(W^T_q,I^n_{\bar{\varepsilon}}\right)\right)^q \ge \left\|Th^\ast\right\|_q^q = \sum\limits_{k=1}^m t_k^qu_k^q+t_{m+1}^qu_{m+1}^q \\ = \sum\limits_{k=1}^m\varepsilon_k^q+t_{m+1}^q\left(1-\sum\limits_{k=1}^m\frac{\varepsilon_{k}^q}{t_k^q}\right)
        =t_{m+1}^q+\sum\limits_{k=1}^m\varepsilon_k^q\left(1-\frac{t_{m+1}^q}{t_k^q}\right),
    \end{gather*}
    which finishes the proof. \qedhere
\end{proof}

\subsection{Information mapping $I^n_{\varepsilon,p}(x)=\left(x_1,\ldots, x_n\right) + B\left[\varepsilon, \ell^n_p\right]$}

We consider three cases separately: $p=\infty$, $p\le q$ and $p>q$.

\subsubsection{Case $p=\infty$}

Setting $\varepsilon_1 = \ldots = \varepsilon_n = \varepsilon$, we obtain from Theorem~\ref{thm_1} the following corollary.

\begin{theorem}
\label{thm_2}
    {\sl Let $n\in\mathbb{N}$, $1\le q < \infty$ and $\varepsilon \ge 0$. If 
    \[
        1 - \varepsilon^q \sum\limits_{k=1}^n \frac{1}{t_k^q} \ge 0
    \]
    then we set $m = n$. Otherwise we choose $m\in\mathcal{Z}_+$, $m\le n$, to be such that
    \[
        1-\varepsilon^q\sum\limits_{k=1}^m\frac{1}{t_k^q}\ge 0 \qquad \text{and}\qquad 1-\varepsilon^q\sum\limits_{k=1}^{m+1}\frac{1}{t_k^q}<0.
    \]
    Then 
    \[
        \mathcal{E}\left(W^T_{q},I^n_{\varepsilon,\infty}\right) = \mathcal{E}\left(W^T_{q},I^n_{\varepsilon,\infty},\Phi_m^*\right) = \left(t_{m+1}^q+\varepsilon^q\sum\limits_{k=1}^m \left(1-\frac{t_{m+1}^q}{t_k^q}\right)\right)^{1/q}.
    \]}
\end{theorem}

\subsubsection{Case $0 < p\le q$}

\begin{theorem}
\label{thm_3}
    {\sl Let $n\in\mathbb{N}$, $1\le q < \infty$ and  $0<p\le q$. If $\varepsilon \in \left[0, t_1\right]$ then
    \[
        \mathcal{E}\left(W^T_q,I^n_{\varepsilon,p}\right) = \mathcal{E}\left(W^T_q,I^n_{\varepsilon,p},\Phi_n^*\right) =\left( {t^q_{n+1}+\varepsilon^q\bigg(1-\frac{t^q_{n+1}}{t^q_{1}}\bigg)}\right)^{1/q},
    \]
    and if $\varepsilon > t_1$ then $\mathcal{E}\left(W^T_q,I^n_ {\varepsilon,p}\right) = \mathcal{E}\left(W^T_q,I^n_ {\varepsilon,p},\Phi_0^*\right) =t_1$. 
    }
\end{theorem}

\begin{proof}[Proof]
    First, consider the case $\varepsilon \in \left[0, t_1\right]$. For $x=Th\in W^T_q$, $\|h\|_q\le 1$, and $a\in I_{\varepsilon,p}^n(x)$, we have
    \begin{gather*}
        \left\|x-\Phi_n^*(a)\right\|^q_q = \sum_{k=1}^n\left|x_k-a_k\bigg(1-\frac{t_{n+1}^q}{t_k^q}\bigg)\right|^q+\sum_{k=n+1}^\infty |x_k|^q \\
        =\sum_{k=1}^n \left|\bigg(1-\frac{t_{n+1}^q}{t_k^q}\bigg)(x_k-a_k)+\frac{t_{n+1}^q}{t_k^q}x_k\right|^q +\sum_{k=n+1}^\infty t_k^q|h_k|^q \\
        \le\sum_{k=1}^n\left(\bigg(1-\frac{t_{n+1}^q}{t_k^q}\bigg)|x_k-a_k|^q+\frac{t_{n+1}^q}{t_k^q}|x_k|^q\right)+t_{n+1}^q\sum_{k=n+1}^\infty |h_k|^q \\
        =\sum_{k=1}^n \bigg( 1-\frac{t_{n+1}^q}{t_k^q}\bigg)\big(|x_k-a_k|^p\big)^{q/p}+ t_{n+1}^q\sum_{k=1}^\infty |h_k|^q \\
        \le \left(1 - \frac{t_{n+1}^q}{t_1^q}\right)\left(\sum\limits_{k=1}^n|x_k-a_k|^p\right)^{q/p} + t_{n+1}^q \le \left(1 - \frac{t_{n+1}^q}{t_1^q}\right)\varepsilon^q + t_{n+1}^q.
    \end{gather*}

    Now, we establish the lower estimate for $\mathcal{E}\left(W_q^T,I^n_{\varepsilon,p}\right)$. Let $u_1$ and $u_{n+1}$ be such that $t_1 u_1=\varepsilon$ and $u_1^q+u_{n+1}^q=1$, {\it i.e.} $u_1={\varepsilon}/{t_1}$ and $u_{n+1}^q= 1-{\varepsilon^q}/{t_1^q}$. Set $h^*:=\left(u_1,0,\ldots,0,u_{n+1},0,\ldots\right)$. Obviously, $\|h\|_q\le 1$ and $\theta\in I_{\varepsilon,p}^n(Th^*)$. Then by Corollary~\ref{cor_identity_operator},
    \begin{gather*}
        \left(\mathcal{E}\left(W_q^T, I^n_{\varepsilon,p}\right)\right)^q\ge \|Th^*\|_q^q=t_{1}^qu_1^q+t_{n+1}^qu_{n+1}^q\\ 
        =\varepsilon^q+t_{n+1}^q\bigg(1-\frac{\varepsilon^q}{t_1^q}\bigg)= t^q_{n+1}+\varepsilon^q\bigg(1-\frac{t^q_{n+1}}{t^q_{1}}\bigg).
    \end{gather*}
    
    Finally, consider the case $\varepsilon > t_1$. For $x = Th\in W^T_q$ and $a\in I^n_{\varepsilon, p}(x)$, we have
    \[
        \left\|x - \Phi_0^*(a)\right\|_q^q = \|Th\|_q^q = \sum\limits_{n=1}^\infty t_n^q|h_n|^q \le t_1^q\sum\limits_{n=1}^\infty |h_n|^q \le t_1^q.
    \]
    Taking $h^* := (1,0,\ldots)$, it is clear that $\theta \in I^n_{\varepsilon, p}\left(Th^*\right)$ and by Corollary~\ref{cor_identity_operator},
    \[
        \mathcal{E}\left(W_q^T, I^n_{\varepsilon,p}\right) \ge \|Th^*\|_q = t_1.
    \]
    Theorem is proved. \qedhere
\end{proof}

\subsubsection{Case $1\le q < p < \infty$}
\label{s321}

This case is the most technical case. We introduce some preliminary notations. For $m=1,\ldots,n$, define 
\[
    \delta_{j,m} := \left(1 - \frac{t_{m+1}^q}{t_j^q}\right)^{\frac{p}{p-q}}, \quad j=1,\ldots,m-1,
\]
and set $c_1 := t_1$ and, for $m\ge 2$,
\begin{equation}
\label{c_m}
    c_{m} := \left(\sum\limits_{j=1}^{m}\delta_{j,m}\right)^{1/p}\left(\sum\limits_{j=1}^{m}\frac{\delta_{j,m}^{q/p}}{t_j^q}\right)^{-1/q}.
\end{equation}
The sequence $\left\{c_m\right\}_{m=1}^n$ is non-increasing. Indeed, let $\delta_{j,m}(\xi) := \left(1 - \frac{\xi t_{m}^q+(1-\xi) t_{m+1}^q}{t_j^q}\right)^{\frac{p}{p-q}}$ and consider the function
\[
    g(\xi) := \left(\sum\limits_{j=1}^{m}\delta_{j,m}(\xi)\right)^{1/p}\left(\sum\limits_{j=1}^{m}\frac{\delta_{j,m}^{q/p}(\xi)}{t_j^q}\right)^{-{1}/{q}},\qquad \xi\in[0,1].
\]
Differentiating $g$ and applying the Cauchy-Swartz inequality we have
\begin{gather*}
    g'(\xi) = \displaystyle \frac{t_{m+1}^q-t_m^q}{p-q}\left(\sum\limits_{j=1}^{m}\delta_{j,m}(\xi)\right)^{\frac{1}{p}-1}\left(\sum\limits_{j=1}^{m}\frac{\delta_{j,m}^{q/p}(\xi)}{t_j^q}\right)^{-{1/q}-1} \times \\
    \displaystyle \times\left(\left(\sum\limits_{j=1}^{m}\frac{\delta_{j,m}^{{q/p}}(\xi)}{t_j^q}\right)^2 - \left(\sum\limits_{j=1}^{m}\delta_{j,m}(\xi)\right)\left(\sum\limits_{j=1}^{m}\frac{\delta_{j,m}^{2{q/p}-1}(\xi)}{t_j^{2q}}\right)\right) \ge 0
\end{gather*}
Hence, $c_{m+1} = g(0) \le g(1) = c_{m}$.

For convenience, for $\lambda\in[0,1]$ denote $t_{m,\lambda}^q := (1- \lambda)t_{m+1}^q + \lambda t_m^q$.

\begin{theorem}
\label{thm_4}
    {\sl Let $n\in\mathbb{N}$ and $1\le q<p<\infty$. 
    \begin{enumerate}
        \item[1.] If $\varepsilon \le c_{n}$ then 
        \[
            \mathcal{E}\left(W^T_q,I^n_{\varepsilon,p}\right) = \mathcal{E}\left(W^T_q,I^n_{\varepsilon,p},\Phi_n^*\right) = \left(t_{n+1}^q + \varepsilon^q \left(\sum\limits_{j=1}^n\left(1 - \frac{t_{n+1}^q}{t_j^q}\right)^{\frac{p}{p-q}}\right)^{\frac{p-q}{p}}\right)^{1/q};
        \]
        \item[2.] If $\varepsilon \in \left(c_n, c_1\right]$ then there exist $m\in\{1,\ldots,n-1\}$ such that $\varepsilon\in \left(c_{m+1},c_m\right]$ and $\lambda = \lambda(\varepsilon)\in[0,1)$ such that
        \begin{equation}
        \label{lambda_condition}
            \varepsilon = \left(\sum\limits_{j=1}^{m}\left(1 - \frac{t_{m,\lambda}^q}{t_j^q}\right)^{\frac{p}{p-q}}\right)^\frac{1}{p}\left(\sum\limits_{j=1}^{m}\frac{\left(1 - \frac{t_{m,\lambda}^q}{t_j^q}\right)^{\frac{q}{p-q}}}{t_j^q}\right)^{-{1/q}}.
        \end{equation}
        Then
        \[
            \mathcal{E}\left(W^T_q, I_{\varepsilon,p}^n\right) = \mathcal{E}\left(W^T_q,I^n_{\varepsilon,p},\Phi_{m,\lambda}^*\right) = \left(t_{m,\lambda}^q + \varepsilon^q\cdot \left(\sum\limits_{j=1}^m \left(1 - \frac{t_{m,\lambda}^q}{t_j^q}\right)^{\frac p{p-q}}\right)^{\frac{p-q}{p}}\right)^{1/q},
        \]
        where
        \[
            \Phi^*_{m,\lambda}(a) = \left(a_1\left(1 - \frac{t_{m,\lambda}^q}{t_1^q}\right),\ldots, a_m\left(1 - \frac{t_{m,\lambda}^q}{t_m^q}\right),0,\ldots\right),\qquad a\in \ell_p.
        \]
        \item[3.] If $\varepsilon > c_1$ then $\mathcal{E}\left(W^T_q,I_{\varepsilon,p}^n\right) = \mathcal{E}\left(W^T_q,I_{\varepsilon,p}^n,\Phi_0^*\right) = t_1$. 
    \end{enumerate}}
\end{theorem}

\begin{proof}[Proof]
    Let $m\in\{0,\ldots,n\}$, $\lambda\in[0,1]$ and $\Phi$ be either $\Phi_n^*$, or $\Phi_0^*$, or $\Phi_{m,\lambda}^*$. For $x\in W^T_q$ and $a\in I^n_{\varepsilon,p}(x)$, 
    \begin{gather*}
        \left\|x - \Phi(a)\right\|_q^q  \le \displaystyle \sum\limits_{k=1}^m \left|\left(1-\frac{t_{m,\lambda}^q}{t_k^q}\right)(x_k-a_k) + \frac{t_{m,\lambda}^q}{t_k^q} x_k\right|^q + \sum\limits_{k=m+1}^\infty |x_k|^q \\
        \le \displaystyle \sum\limits_{k=1}^m \left(1 - \frac{t_{m,\lambda}^q}{t_k^q}\right)\left|x_k - a_k\right|^q + \sum\limits_{k=1}^m t_{m,\lambda}^q |h_k|^q + \sum\limits_{k=m+1}^\infty t_k^q|h_k|^q. 
    \end{gather*}
    Using the H\"older inequality with parameters ${p}/{(p-q)}$ and ${p}/{q}$ to estimate the first term and inequality $t_k^q\le t_{m,\lambda}^q$, $k=m+1,m+2,\ldots$, we obtain
    \begin{gather*}
        \left\|x - \Phi(a)\right\|_q^q \le \displaystyle \left\{\sum\limits_{k=1}^m\left(1 - \frac{t_{m,\lambda}^q}{t_k^q}\right)^{\frac{p}{p-q}}\right\}^{1-{q}/{p}}\left\{\sum\limits_{k=1}^m|x_k-a_k|^p\right\}^{q/p} + t_{m,\lambda}^q \sum\limits_{k=1}^\infty h_k^q\\
        \le \displaystyle \left\{\sum\limits_{k=1}^m\left(1 - \frac{t_{m,\lambda}^q}{t_k^q}\right)^{\frac{p}{p-q}}\right\}^{1-{q}/{p}}\varepsilon^q + t_{m,\lambda}^q,
    \end{gather*}
    which proves the estimate from above.
    
    Now, we turn to the proof of the lower estimate. First, let $\varepsilon \le c_{n}$, and define 
    \[
        u_j := \frac{\varepsilon\,\delta_{j,n}^{1/p}}{t_j}\left(\sum\limits_{j=1}^n\delta_{j,n}\right)^{-1/p}, \quad j=1,\ldots,n,\qquad\text{and}\qquad u_{n+1}:=\left(1 - \sum\limits_{j=1}^n u_j^q\right)^{1/q}.
    \]
    Consider $h^* :=\left(u_1,\ldots,u_{n+1},0,\ldots\right)$. Evidently, $u_{n+1}$ is well-defined as
    \[
        \sum\limits_{j=1}^n u_j^q = \varepsilon^q \left(\sum\limits_{j=1}^n\delta_{j,n}\right)^{-q/p} \sum\limits_{j=1}^n \frac{\delta_{j,n}^{q/p}}{t_j^q} = \frac{\varepsilon^q}{c^q_{n}} \le 1, 
    \]
    $\|h^*\|_q = 1$ and $\theta\in I_{\varepsilon,p}^n(Th^*)$ as  $\sum_{j=1}^n t_j^ph_j^p = \varepsilon^p$. Hence, by Corollary~\ref{cor_identity_operator},
    \begin{gather*}
        \left(\mathcal{E}\left(W^T_q,I^n_{\varepsilon,p}\right)\right)^q \ge \displaystyle \left\|Th^*\right\|_q^q \\ 
        = \varepsilon^q \sum\limits_{j=1}^n \delta_{j,n}^{q/p} \left(\sum\limits_{j=1}^n\delta_{j,n}\right)^{-q/p} + t_{n+1}^q - t_{n+1}^q\sum\limits_{j=1}^n \varepsilon^q\frac{\delta_{j,n}^{q/p}}{t_j^q}\left(\sum\limits_{j=1}^n\delta_{j,n}\right)^{-{q}/{p}} \\ 
        = \displaystyle \varepsilon^q\left(\sum\limits_{j=1}^n\delta_{j,n}\right)^{-q/p} \sum\limits_{j=1}^n \delta_{j,n}^{q/p} \left(1 - \frac{t_{n+1}^q}{t_j^q}\right) + t_{n+1}^q \\
        = \displaystyle t_{n+1}^q + \varepsilon^q\cdot \left(\sum\limits_{j=1}^n\left(1 - \frac{t_{n+1}^q}{t_j^q}\right)^{\frac{p}{p-q}}\right)^{\frac{p-q}{p}}.
    \end{gather*}
    
    Next, let $m\in\{1,2,\ldots,n-1\}$ be such that $c_{m+1}<\varepsilon\le c_{m}$ and $\lambda = \lambda_\varepsilon\in[0,1)$ be defined by~\eqref{lambda_condition}. Set 
    \[
        u_j := \frac{\varepsilon \delta^{1/p}_{j,m}(\lambda)}{t_j}\left(\sum\limits_{j=1}^m\delta_{j,m}(\lambda)\right)^{-1/p},\quad j=1,\ldots,m, 
    \]
    and consider $h^* = \left(u_1,\ldots,u_m,0,\ldots\right)$. Clearly, $\|h\|_q = 1$ and $\theta\in I_{\varepsilon,p}^n(Th^*)$. Using Corollary~\ref{cor_identity_operator}, we obtain the desired lower estimate for $\mathcal{E}\left(W^T_q,I^n_{\varepsilon,p}\right)$. 
    
    Finally, let $\varepsilon > c_1$. Consider $h^* := \left(1,0,0,\ldots\right)$. Since $c_1 = t_1$, we have $\theta\in I_{\varepsilon,p}^n(Th^*)$. Hence, by Corollary~\ref{cor_identity_operator}, $\mathcal{E}\left(W^T_q,I^n_{\varepsilon,p}\right) \ge \left\|Th^\ast\right\|_q = t_1^q$.
\end{proof}

\subsection{Information mapping $I(x) = I_{\varepsilon,p}(x) := x + B\left[\varepsilon, \ell_p\right]$}

As a limiting case from Theorem~\ref{thm_2},~\ref{thm_3} and~\ref{thm_4} we can obtain the following corollaries. 

\begin{theorem}
\label{thm5}
    {\sl Let $1\le q < \infty$ and $\varepsilon\ge 0$. Choose $m\in\mathbb{Z}_+$ to be such that
    \[
        1 - \varepsilon^q\sum\limits_{k=1}^m\frac{1}{t_k^q} \ge 0\qquad\textrm{and}\qquad 1 - \varepsilon^q\sum\limits_{k=1}^{m+1}\frac{1}{t_k^q} < 0.
    \]
    Then
    \[
        \mathcal{E}\left(W^T_q,I_{\varepsilon,\infty}\right) = \mathcal{E}\left(W^T_q,I_{\varepsilon,\infty},\Phi_m^*\right) = \left(t_{m+1}^q + \varepsilon^q \sum\limits_{k=1}^m\left(1 - \frac{t_{m+1}^q}{t_k^q}\right)\right)^{1/q}.
    \]}
\end{theorem}

\begin{theorem}
    {\sl Let $1\le q < \infty$ and $0 < p \le q$. If $0 \le \varepsilon \le t_1$ then $\mathcal{E}\left(W^T_q,I_{\varepsilon,p}\right) = \mathcal{E}\left(W^T_q,I_{\varepsilon,p},{\rm id}\right) = \varepsilon$, and if $\varepsilon > t_1$ then $\mathcal{E}\left(W^T_q,I_{\varepsilon,p}\right) = \mathcal{E}\left(W^T_q,I_{\varepsilon,p},\Phi_0^*\right) = t_1$.}
\end{theorem}

Define the sequence $\left\{c_n\right\}_{n=1}^\infty$ using formulas~\eqref{c_m}. It is not difficult to verify that $\left\{c_n\right\}_{n=1}^\infty$ is non-increasing and tend to $0$ as $n\to\infty$.

\begin{theorem}
\label{thm7}
    {\sl Let $1\le q < p < \infty$. If $\varepsilon \in \left(0, c_1\right]$ then there exists $m\in\mathbb{N}$ such that $\varepsilon\in \left(c_{m+1},c_m\right]$ and $\lambda = \lambda(\varepsilon)\in[0,1)$ such that
        \begin{equation}
        \label{lambda_condition1}
            \varepsilon = \left(\sum\limits_{j=1}^{m}\left(1 - \frac{t_{m,\lambda}^q}{t_j^q}\right)^{\frac{p}{p-q}}\right)^{1/p}\left(\sum\limits_{j=1}^{m}\frac{\left(1 - \frac{t_{m,\lambda}^q}{t_j^q}\right)^{\frac{q}{p-q}}}{t_j^q}\right)^{-{1/q}}.
        \end{equation}
        Then
        \[
            \mathcal{E}\left(W^T_q, I_{\varepsilon,p}\right) = \mathcal{E}\left(W^T_q,I_{\varepsilon,p},\Phi_{m,\lambda}^*\right) = \left(t_{m,\lambda}^q + \varepsilon^q\cdot \left(\sum\limits_{j=1}^m \left(1 - \frac{t_{m,\lambda}^q}{t_j^q}\right)^{\frac p{p-q}}\right)^{\frac{p-q}{p}}\right)^{1/q}.
        \]
        where the method $\Phi^*_{m,\lambda}$ is defined in Theorem~\ref{thm_3}. Otherwise, if $\varepsilon > c_1$ then $\mathcal{E}\left(W^T_q,I_{\varepsilon,p}\right) = \mathcal{E}\left(W^T_q,I_{\varepsilon,p},\Phi_0^*\right) = t_1$. }
\end{theorem}

\section{Recovery of scalar products}

    Following~\cite{Bab_87} (see also~\cite{Bab_88, BabGunRud_12, BabGunPar_20}), let us consider the problem of optimal recovery of scalar product. Let $1\le p,q\le \infty$ and given operators $T: \ell_p\to \ell_p$ and $S: \ell_q\to \ell_q$ be defined as follows: for fixed non-increasing sequences $t=\{t_k\}_{k=1}^{\infty}$ and $s=\{s_k\}_{k=1}^{\infty}$,
    \[
        Th:=\{t_kh_k\}_{k=1}^{\infty}, \quad h\in \ell_p,\qquad\text{and}\qquad Sg:=\{s_kg_k\}_{k=1}^{\infty}, \quad g\in \ell_q.
    \]
    Consider classes of sequences
    \begin{gather*}
        W_p^T:=\left\{x=Th\,:\,h\in \ell_p,\, \|h\|_p\le 1\right\},\\
        W_q^S:=\left\{y=Tg\,:\, g\in \ell_q,\, \|g\|_q\le 1\right\}.
    \end{gather*}
    and define the scalar product $A = \left<\cdot,\cdot\right>:\ell_p\times\ell_q\to\mathbb{C}$ as usually: 
    \[
        \left<x,y\right>=\sum\limits_{k=1}^{\infty}x_ky_k,\qquad x\in \ell_p,\; y\in \ell_q.
    \]
    For brevity, we denote $\left<x,y\right>_n := \sum_{k=1}^n x_ky_k$.

    In this section we will consider the problem of optimal recovery of scalar product operator $A$ on the class $W_{p,q}^{T,S} := W_p^T\times W_q^S$, when information mapping $I$ is given in one of the following forms:
    \begin{enumerate}
        \item $I(x,y) = J_{\bar{\varepsilon}}^n(x, y)= \left\{(a, b)\in\mathbb{C}^n\times\mathbb{C}^n\,:\, \forall k=1,...,n \Rightarrow |x_ky_k - a_kb_k|\le\varepsilon_k\right\}$, where $n\in\mathbb{N}$ and $\varepsilon_1,\ldots,\varepsilon_n \ge 0$;
        \item $I(x,y) = J^n_{\varepsilon, r}(x, y) = \left\{(a,b)\in\mathbb{C}^n\times\mathbb{C}^n\,:\,\left\|\left<x,y\right>_n - \left<a,b\right>_n\right\|_{\ell_r^n}\le \varepsilon\right\}$, where $n\in\mathbb{N}$ and $1\le r\le \infty$.
    \end{enumerate}
    
    Finally, for $m\in\mathbb{N}$, we define methods of recovery $\Psi_m^*:\mathbb{C}^n\times\mathbb{C}^n\to \mathbb{C}$: 
    \[
        \Psi_m^*(a, b) = \sum_{k=1}^{m} a_kb_k\bigg(1-\frac{t_{m+1}s_{m+1}}{t_{k}s_k}\bigg),\qquad a,b\in\mathbb{C}^n,
    \]
    that will be optimal in many situations and set $\Psi_0^*(a,b) := 0$.

\subsection{Information mapping $J^n_{\bar{\varepsilon}}$}


\begin{theorem}
\label{th5.1}
    {\sl 
    Let $n\in\mathbb{N}$, $1< p < \infty$, $q=p/(p-1)$, $\varepsilon_1,\ldots,\varepsilon_n \ge 0$ and $\bar{\varepsilon} = \left(\varepsilon_1,\ldots,\varepsilon_n\right)$. If 
    \[
        1 - \sum\limits_{k=1}^n\frac{\varepsilon_k}{t_ks_k} \ge 0
    \]
    we set $m=n$. Otherwise we choose $m\in\mathbb{Z}_+$, $m\le n$, to be such that
    \[
        1-\sum\limits_{k=1}^m\frac{\varepsilon_k}{t_ks_k}\ge 0\qquad \textrm{and} \qquad 1-\sum\limits_{k=1}^m\frac{\varepsilon_k}{t_ks_k}-\frac{\varepsilon_{m+1}}{t_{m+1}s_{m+1}} < 0.
    \]
    Then
    \[
        \mathcal{E}\left(A,W_{p,q}^{T,S}, J^n_{\bar{\varepsilon}}\right)= \mathcal{E}\left(A, W_{p,q}^{T,S},J^n_{\bar\varepsilon},\Psi_m^*\right) =  t_{m+1}s_{m+1}+\sum\limits_{k=1}^m\varepsilon_k\bigg(1-\frac{t_{m+1}s_{m+1}}{t_{k}s_k}\bigg).
    \]}
\end{theorem}

\begin{proof}[Proof] Using the triangle inequality, relations $\left|x_ky_k - a_kb_k\right|\le \varepsilon_k$, $k=1,\ldots,n$, and monotony of sequences $t$ and $s$, we obtain that, for $(x, y) = (Th, Sg)\in W_{p,q}^{T,S}$ and $(a, b) \in J_{\bar{\varepsilon}}^n(x, y)$, 
    \begin{gather*}
        \left|\left<x,y\right>-\Psi_m^*(a, b)\right|= \left|\sum_{k=1}^\infty x_k y_k - \sum_{k=1}^m a_kb_k\left(1-\frac{t_{m+1}s_{m+1}}{t_k s_k}\right)\right| \\
        \le \sum_{k=1}^m \left(1-\frac{t_{m+1}s_{m+1}}{t_ks_k}\right)|x_k y_k - a_kb_k|+\sum_{k=1}^m \frac{t_{m+1}s_{m+1}}{t_ks_k}|x_k y_k|+\sum_{k=m+1}^\infty |x_k y_k|\\
        \le\sum_{k=1}^m \left(1-\frac{t_{m+1}s_{m+1}}{t_ks_k}\right)\varepsilon_k+t_{m+1}s_{m+1}\sum_{k=1}^m |h_k g_k|+\sum_{k=m+1}^\infty t_ks_k|h_k g_k|\\
        \le\sum_{k=1}^m \left(1-\frac{t_{m+1}s_{m+1}}{t_ks_k}\right)\varepsilon_k+t_{m+1}s_{m+1}\sum_{k=1}^\infty |h_k g_k|\\
        \le t_{m+1}s_{m+1}+\sum_{k=1}^m \left(1-\frac{t_{m+1}s_{m+1}}{t_ks_k}\right)\varepsilon_k,
    \end{gather*}
    which proves the upper estimate.

    To establish the lower estimate, we set
    \begin{gather}
        u_k := \left(\frac{\varepsilon_k}{t_ks_k}\right)^{1/p},\qquad v_k := \left(\frac {\varepsilon_k}{t_ks_k}\right)^{1/q},\qquad k=1,\ldots,m,\\
        u_{m+1}= \left(1 - \sum\limits_{k=1}^m\frac{\varepsilon_k}{t_ks_k}\right)^{1/p},\qquad v_{m+1}=\left(1-\sum\limits_{k=1}^m\frac {\varepsilon_k}{t_ks_k}\right)^{1/q},
    \end{gather}
    and consider $u^*=\left(u_1,\ldots,u_{m+1},0,\ldots\right)$ and $v^*=\left(v_1,\ldots,v_{m+1},0,\ldots\right)$. It is clear that $\left(Tu^*, Sv^*\right)\in W_{p,q}^{T,S}$ and $\theta\in J_{\bar{\varepsilon}}^n(Tu^*, Sv^*)\cap J_{\bar{\varepsilon}}^n(-Tu^*, Sv^*)$ due to the choice of number $m$. Hence, by Corollary~\ref{cor_scalar_product}, 
    \begin{gather*}
        E\left(A, W_{p,q}^{T,S}, J_{\bar{\varepsilon}}^n\right) \ge \left|\left<Tu^*, Sv^*\right>\right|= \sum_{k=1}^m t_ks_ku_kv_k+t_{m+1}s_{m+1}\left(1-\sum_{k=1}^m\frac{\varepsilon_k}{t_ks_k}\right)\\
        =t_{m+1}s_{m+1}+\sum_{k=1}^m\left(1-\frac{t_{m+1}s_{m+1}}{t_ks_k}\right)\varepsilon_k.
    \end{gather*}
    This proves the sharpness of the upper estimate. Theorem is proved. \qedhere
\end{proof}

\subsection{Information mapping $J^n_{\varepsilon,r}$}

We consider three cases separately: $r=\infty$, $0 < r\le 1$ and $1<r<\infty$.

\subsubsection{Case $r = \infty$}

Setting $\varepsilon_1 = \ldots = \varepsilon_n = \varepsilon$, we obtain the following corollary from Theorem~\ref{th5.1}.

\begin{theorem}
    {\sl Let $n\in\mathbb{N}$, $1< p< \infty$, $q = p/(p-1)$ and $\varepsilon\ge 0$. If 
    \[
        1 - \varepsilon\sum\limits_{k=1}^n\frac{1}{t_ks_k}\ge 0,
    \]
    we set $n = m$. Otherwise we choose $m\in\mathbb{Z}_+$, $m\le n$, to be such that
    \[
        1-\varepsilon\sum\limits_{k=1}^m\frac{1}{t_ks_k}\ge 0\qquad \textrm{and} \qquad 1-\varepsilon\sum\limits_{k=1}^{m+1}\frac{1}{t_ks_k}< 0.
    \]
    Then
    \[
        \mathcal{E}\left(A,W_{p,q}^{T,S},J^n_{\varepsilon,\infty}\right) = \mathcal{E}\left(A,W_{p,q}^{T,S},J^n_{\varepsilon,\infty},\Psi^*_m\right) = t_{m+1}s_{m+1}+\varepsilon\sum_{k=1}^m\left(1-\frac{t_{m+1}s_{m+1}}{t_ks_k}\right).
    \]}
\end{theorem}

\subsubsection{Case $0 < r \le 1$}

\begin{theorem}
    {\sl Let $n\in\mathbb{N}$, $1< p< \infty$, $q = p/(p-1)$ and $r\in (0,1]$. If $\varepsilon \le t_1s_1$ then
    \[
        \mathcal{E}\left(A, W_{p,q}^{T,S},J^n_{\varepsilon,r}\right) = \mathcal{E}\left(A, W_{p,q}^{T,S},J^n_{\varepsilon,r},\Psi^*_n\right) = t_{n+1}s_{n+1}+\varepsilon\left(1-\frac{t_{n+1}s_{n+1}}{t_1s_1}\right),
    \]
    and if $\varepsilon > t_1s_1$ then $\mathcal{E}\left(A,W_{p,q}^{T,S},J^n_{\varepsilon,r}\right) = \mathcal{E}\left(A,W_{p,q}^{T,S},J^n_{\varepsilon,r},\Psi^*_0\right) = t_{1}s_{1}$.}
\end{theorem}

\begin{proof}[Proof]
    First, we consider the case $\varepsilon \le t_1s_1$. Let $(x,y) = (Th,Sg)\in W^{T,S}_{p,q}$ and $(a,b)\in J^n_{\varepsilon,r}(x,y)$. Similarly to the proof of Theorem~\ref{th5.1} in the case $m=n$ we obtain
    \begin{equation}
    \label{ocenka ukloneniia}
        \left|\left<x,y\right>-\Psi^*_n (a,b)\right| \le\sum_{k=1}^n\left(1-\frac{t_{n+1}s_{n+1}}{t_ks_k}\right)|x_ky_k-a_kb_k|+t_{n+1}s_{n+1}\sum_{k=1}^\infty  h_kg_k.
    \end{equation}
    Using the H\"older inequality and inequality $\varepsilon_1^{1/r} + \ldots + \varepsilon_n^{1/r}\le \left(\varepsilon_1+\ldots+\varepsilon_n\right)^{1/r}$, we have
    \begin{gather*}
        \displaystyle \left|\left<x,y\right>-\Psi^*_n (a,b)\right| \le \displaystyle \max_{k=\overline{1,n}}\left(1-\frac{t_{n+1}s_{n+1}}{t_ks_k}\right)\sum_{k=1}^n|x_ky_k-a_kb_k| + t_{n+1}s_{n+1}\|h\|_p \|g\|_q \\ 
        \le \displaystyle  \left(1-\frac{t_{n+1}s_{n+1}}{t_1s_1}\right)\varepsilon+ t_{n+1}s_{n+1}.
    \end{gather*}
    The upper estimate is proved.

    Now, we establish the lower estimate. Let 
    \[
        u_1=\left(\frac{\varepsilon}{s_1t_1}\right)^{1/p},\quad u_{n+1}=\left(1-\frac{\varepsilon}{t_1s_1}\right)^{1/p},\quad v_1=\left(\frac{\varepsilon}{s_1t_1}\right)^{1/q},\quad v_{n+1}=\left(1-\frac{\varepsilon}{t_1s_1}\right)^{1/q}, 
    \]
    and consider elements $u^* = \left(u_1,0,\ldots,0,u_{n+1},0,\ldots\right)$, $v^*=\left(u_1,0,\ldots,0,u_{n+1},0,\ldots\right)$. Obviously, $\left(Tu^*,Sv^*\right)\in W^{T,S}_{p,q}$ and $(\theta,\theta)\in J^n_{\varepsilon,r}(Tu^*,Sv^*)\cap J^n_{\varepsilon,r}(-Tu^*,Sv^*)$. Then by Corollary~\ref{cor_scalar_product},
    \begin{gather*}
        \mathcal{E}\left(A, W_{p,q}^{T,S},J_{\varepsilon,r}^n\right) \ge \left|\left<Tu^*, Sv^*\right>\right| = t_1s_1\cdot\frac{\varepsilon}{t_1s_1} + t_{n+1}s_{n+1}\cdot\left(1 - \frac{\varepsilon}{t_1s_1}\right) \\
        = \left(1 - \frac{t_{n+1}s_{n+1}}{t_1s_1}\right)\varepsilon + t_{n+1}s_{n+1},
    \end{gather*}
    which finishes the proof of the desired estimate.

    Next, we let $\varepsilon >t_1s_1$. For $(x,y)=(Th,Sg)\in W^{T,S}_{p,q}$, and $(a,b)\in J^n_{\varepsilon,r}(x,y)$, we have
    \[
        \left|\left<x,y\right>-\Psi^*_0(a,b)\right|=|\left<x,y\right>|\le \sum_{k=1}^\infty t_ks_k|h_kg_k|\le t_1s_1 \|h\|_p \|g\|_q\le t_1s_1.
    \]
    Taking $u^*=v^*=(1,0,...)$, it is clear that $(\theta,\theta)\in J^n_{\varepsilon,r}(Tu^*,Sv^*)\cap J^n_{\varepsilon,r}(-Tu^*,Sv^*)$. By Corollary~\ref{cor_scalar_product}, 
    \[
        \mathcal{E}\left(A,W_{p,q}^{T,S},J^n_{\varepsilon,r}\right)\ge \left|\left(Tu^*,Sv^*\right)\right|=t_1s_1.
    \]
    The theorem is proved. \qedhere
\end{proof}

\subsubsection{Case $1< r< \infty$}

First, we introduce some preliminary notations. For $m=1,\ldots,n$, we define 
\[
    \tau_{j,m} := \left(1 - \frac{t_{m+1}s_{m+1}}{t_js_j}\right)^{\frac{1}{r-1}}, \quad j=1,\ldots,m-1,
\]
and set $d_1:=t_1s_1$ and, for $m\ge 2$,
\[
    d_m := \left(\sum\limits_{j=1}^{m}\tau^r_{j,m}\right)^{1/r}\left(\sum\limits_{j=1}^{m}\frac{\tau_{j,m}}{t_js_j}\right)^{-1}.
\]
The sequence $\left\{d_m\right\}_{m=1}^n$ is non-increasing, which can be verified using the arguments similar to those applied to prove monotony of sequence $\left\{c_m\right\}_{m=1}^n$ in subsection~\ref{s321}. In addition, for convenience, for $\lambda\in[0,1]$, we denote 
\[
    t_{m,\lambda} := (1-\lambda)t_{m+1} + \lambda t_{m} \qquad\text{and}\qquad s_{m,\lambda} := (1-\lambda)s_{m+1} + \lambda s_{m}.
\]

\begin{theorem}
    {\sl Let $n\in\mathbb{N}$, $1< p <\infty$, $q=p/(p-1)$ and $1< r < \infty$. 
    \begin{enumerate}
        \item[1.] If $\varepsilon \le d_{n+1}$ then
        \[
            \mathcal{E}\left(A,W^{T,S}_{p,q},J^n_{\varepsilon,r}\right) = \mathcal{E}\left(A,W^{T,S}_{p,q},J^n_{\varepsilon,r},\Psi^*_{n}\right) = t_{n+1}s_{n+1} + \varepsilon\cdot \left(\sum\limits_{j=1}^n\left(1 - \frac{t_{n+1}s_{n+1}}{t_js_j}\right)^{\frac{r}{r-1}}\right)^{\frac{r-1}{r}}.
        \]
        \item[2.] If $\varepsilon\in\left(d_n,d_1\right]$ then there exist $m\in\{1,\ldots,n-1\}$ such that $\varepsilon\in \left(d_{m+1}, d_m\right]$ and $\lambda = \lambda(\varepsilon)\in[0,1)$ such that
        \begin{equation}
        \label{lambda_cond_2}
            \varepsilon = \left(\sum\limits_{j=1}^{m}\left(1 - \frac{t_{m,\lambda}s_{m,\lambda}}{t_js_j}\right)^{\frac{r}{r-1}}\right)^{1/r}\left(\sum\limits_{j=1}^{m}\frac{\left(1 - \frac{t_{m,\lambda}s_{m,\lambda}}{t_js_j}\right)^{\frac{1}{r-1}}}{t_js_j}\right)^{-1}.
        \end{equation}
        Then
        \[
            \mathcal{E}\left(A,W^{T,S}_{p,q},J_{\varepsilon,r}^n\right) = \mathcal{E}\left(A,W^{T,S}_{p,q},J_{\varepsilon,r}^n,\Psi^*_{m,\lambda}\right) = t_{m,\lambda}s_{m,\lambda} + \varepsilon \left(\sum\limits_{j=1}^m \left(1 - \frac{t_{m,\lambda}s_{m,\lambda}}{t_js_j}\right)^{\frac{r}{r-1}}\right)^{\frac{r-1}r},
    \]
    where
    \[
        \Psi^*_{m,\lambda}\left(a,b\right) = \sum\limits_{j=1}^m a_jb_j\left(1 - \frac{t_{m,\lambda}s_{m,\lambda}}{t_js_j}\right),\qquad a,b\in \ell_r.
    \]
    \item[3.] If $\varepsilon > d_1$ then $\mathcal{E}\left(A,W^{T,S}_{p,q},J_{\varepsilon,r}^n\right) = \mathcal{E}\left(A,W^{T,S}_{p,q},J_{\varepsilon,r}^n,\Psi^*_0\right) = t_1s_1$.
    \end{enumerate}}
\end{theorem}

\begin{proof}[Proof]
    Let $m\in\{0,\ldots,n\}$, $\lambda\in[0,1]$ and $\Psi$ be either $\Psi_n^*$ or $\Psi_0^*$, or $\Psi_{m,\lambda}^*$. Using the H\"older inequality with parameters $r$ and $\frac{r}{r-1}$, for $(x,y) = (Th,Sg)\in W^{T,S}_{p,q}$ and $\left(a,b\right)\in J_{\varepsilon,r}^n(x,y)$, we have
    \begin{gather*}
        \displaystyle \left|\left<x,y\right> - \Psi\left(a,b\right)\right| \le \displaystyle \sum\limits_{j=1}^m \left(1 - \frac{t_{m,\lambda}s_{m,\lambda}}{t_js_j}\right) \left|x_jy_j - a_jb_j\right| + t_{m,\lambda}s_{m,\lambda}\sum\limits_{j=1}^m \frac{x_jy_j}{t_js_j} + \sum\limits_{j=m+1}^\infty x_jy_j\\ 
        \le \displaystyle t_{m,\lambda}s_{m,\lambda} + \varepsilon \left(\sum\limits_{j=1}^m \left(1 - \frac{t_{m,\lambda}s_{m,\lambda}}{t_js_j}\right)^{\frac{r}{r-1}}\right)^{\frac{r-1}r},
    \end{gather*}
    which proves the upper estimate. 
    
    Now, we turn to the proof of the lower estimate. We let $\varepsilon \le d_{n}$, and, for $j=1,\ldots,n$, set
    \[
        u_j = \left(\frac{\varepsilon \tau_{j,n}}{t_js_j}\right)^{1/p}\left(\sum\limits_{k=1}^n\tau_{k,n}^r\right)^{-\frac{1}{rp}},\qquad v_j = \left(\frac{\varepsilon \tau_{j,n}}{t_js_j}\right)^{1/q}\left(\sum\limits_{k=1}^n\tau_{k,n}^r\right)^{-\frac{1}{rq}},
    \]
    $u_{n+1} := \left(1 - u_1^p-\ldots-u_n^p \right)^{1/p}$, and $v_{n+1} := \left(1 - v_1^q-\ldots-v_n^q\right)^{1/q}$. In addition, we define
    $u^* = \left(u_1,\ldots,u_n,u_{n+1},0,\ldots\right)$ and $v^* := \left(v_1,\ldots,v_n, v_{n+1},0,\ldots\right)$.    By the choice of $\varepsilon$, numbers $u_{n+1}$ and $v_{n+1}$ are well defined and, hence, $(Tu^*,Sv^*)\in W^{T,S}_{p,q}$. Also, 
    \[
        \sum\limits_{j=1}^n \left|Tu^*_j\cdot Sv^*_j\right|^r = \sum\limits_{j=1}^n \left|t_js_ju_jv_j\right|^r = \varepsilon^r,
    \]
    yielding that $\left(\theta,\theta\right) \in J_{\varepsilon,r}^n\left(Tu^*,Sv^*\right)\cap J_{\varepsilon,r}^n(-Tu^*,Sv^*)$. By Corollary~\ref{cor_scalar_product},
    \begin{gather*}
        \displaystyle \mathcal{E}\left(A,W^{T,S}_{p,q}, J_{\varepsilon,r}^n\right) \ge \displaystyle \left|\left(Tu^*, Sv^*\right)\right| \\ 
        = \varepsilon \sum\limits_{j=1}^n \tau_{j,n} \left(\sum\limits_{j=1}^n\tau_{j,n}^r\right)^{-1/r} + t_{n+1}s_{n+1} - \varepsilon \sum\limits_{j=1}^n \frac{t_{n+1}s_{n+1}\tau_{j,n}}{t_js_j} \left(\sum\limits_{j=1}^n\tau_{j,n}^r\right)^{-1/r}\\
        =\displaystyle t_{n+1}s_{n+1} +\varepsilon \left(\sum\limits_{j=1}^n \left(1 - \frac{t_{n+1}s_{n+1}}{t_js_j}\right)^{\frac{r}{r-1}}\right)^{\frac{r-1}r},
    \end{gather*}
    which proves the desired lower estimate. 
    
    Next, let $m\in\{1,\ldots,n-1\}$ be such that $d_{m+1} < \varepsilon\le d_{m}$ and $\lambda=\lambda_\varepsilon\in[0,1)$ be defined by~\eqref{lambda_cond_2}. Set
    \[
        u_j := \left(\frac{\varepsilon\tau_{j}}{t_js_j}\right)^{1/p} \left(\sum\limits_{k=1}^m \tau_k^r\right)^{-\frac 1{rp}}\quad\text{and}\quad v_j := \left(\frac{\varepsilon\tau_{j}}{t_js_j}\right)^{1/q} \left(\sum\limits_{k=1}^m \tau_k^r\right)^{-\frac 1{rq}},
    \]
    and define $u^* := \left(u_1,\ldots,u_m,0,\ldots\right)$ and $v^* := \left(v_1,\ldots,v_m,0,\ldots\right)$. It is not difficult to verify that $\left(\theta,\theta\right) \in J_{\varepsilon,r}^n\left(Tu^*,Sv^*\right)\cap J_{\varepsilon,r}^n(-Tu^*,Sv^*)$. Using Corollary~\ref{cor_scalar_product} we obtain the desired estimate for $\mathcal{E}\left(A,W^{T,S}_{p,q},J^n_{\varepsilon,r}\right)$. 
    
    Finally, let $\varepsilon > t_1s_1$. Consider $u^* = v^* = \left(1,0,\ldots\right)$. Since $d_1 = t_1s_1$, we have $(\theta, \theta)\in J^n_{\varepsilon,r}(Tu^*,Sv^*)\cap J^n_{\varepsilon,r}(-Tu^*,Sv^*)$. Hence, by Corollary~\ref{cor_scalar_product}, $\mathcal{E}\left(A,W^{T,S}_{p,q},J_{\varepsilon,r}^n\right) \ge \left|\left<Tu^*, Sv^*\right>\right| = t_1s_1$.
\end{proof}

\subsection{Applications}

Let $H$ be a complex Hilbert space with orthonormal basis $\left\{\varphi_n\right\}_{n=1}^\infty$, $\{t_k\}_{k=1}^\infty$ be a non-increasing sequence; $T:\ell_2\to\ell_2$ be an operator mapping sequence $x = \left(x_1,x_2,\ldots\right)$ into sequence $Tx = \left(t_1x_1, t_2x_2,\ldots\right)$. Consider the class 
\[
    \mathcal{W}^T := \left\{x = \sum\limits_{n=1}^\infty t_nc_n \varphi_n\,:\,\sum\limits_{n=1}^\infty \left|c_n\right|^2 \le 1\right\},
\]
and information operator $\mathcal{I}_{p,\varepsilon}:H\to \ell_p$, with $2 <p \le \infty$, mapping an element $x = \sum_{n=1}^\infty x_n\varphi_n$ into the set $\mathcal{I}_{p,\varepsilon}x = \left(x_1,x_2,\ldots\right) + B[\varepsilon,\ell_p]\in\ell_p$. Due to isomorphism between $\ell_2$ and $H$, under notations of Section~\ref{sec_3} we have
\begin{equation}
\label{identity}
    \mathcal{E}\left(\mathcal{W}^T,\mathcal{I}_{\varepsilon,p}\right) = \mathcal{E}\left(W^T_2,I^\infty_{\varepsilon,p}\right).
\end{equation}
Moreover, methods of recovery $F_{m,\lambda}^* := \mathfrak{A}\circ\Phi^*_{m,\lambda}$ are optimal, where $\mathfrak{A}:\ell_2\to H$ is the natural isomorphism between $\ell_2$ and $H$: $\mathfrak{A}\left(x_1,x_2,\ldots\right) = \sum_{n=1}^\infty x_n \varphi_n$. Remark that $F_{m,\lambda}$ are triangular methods of recovery that play an important role in the theory of ill-posed problems (see, {\it e.g.}~\cite[Theorem~2.1]{MatPer_02} and references therein).

Consider an important case when $t_n = n^{-\mu}$, $n\in\mathbb{N}$, with some fixed $\mu > 0$. It corresponds {\it e.g.}, to the space $H = L_2(\mathbb{T})$ of square integrable functions defined on a period and the class $\mathcal{W}^T = W^\mu_2(\mathbb{T})$ of functions having $L_2$-bounded Weyl derivative of order $\mu$. Using equality~\eqref{identity} and Theorems~\ref{thm7} and~\ref{thm5}, we obtain
\[
    \lim\limits_{\varepsilon\to 0^+}\varepsilon^{-\lambda}\mathcal{E}\left(\mathcal{W}^T,\mathcal{I}_{\varepsilon,p}\right) = \left(\frac{\alpha+\beta}{\beta}\right)^{1/2} \left(\frac{\beta^{1/2}}{\alpha^{1/p}}\right)^{\lambda},\qquad \lambda = \frac{\mu}{\mu + 1/2 - 1/p},
\]
where
\[
    \alpha = \frac{1}{2\mu}B\left(\frac{2p-2}{p},\frac{1}{2\mu}\right),\qquad \beta = \frac{1}{2\mu}B\left(\frac{2p-2}{p},2+\frac{1}{2\mu}\right),
\]
and $B(\alpha,\beta)$ is the Euler beta function. Indeed, in case $2 < p < \infty$, by selecting $n = n_\varepsilon\in \mathbb{N}$ and $\lambda\in[0,1)$ such that equation~\eqref{lambda_condition1} is satisfied, we can easily verify that
\[
    \lim\limits_{\varepsilon\to 0^+} n_\varepsilon^{\mu/\lambda} c_{n_\varepsilon} = \lim\limits_{\varepsilon\to 0^+} n_\varepsilon^{\mu/\lambda} c_{n_\varepsilon+1} = \alpha^{1/p}\beta^{-1/2}\quad\text{and}\quad \lim\limits_{\varepsilon\to 0^+} \varepsilon^{-\mu/\lambda} n_\varepsilon^{-\mu} = \left(\frac{\beta^{1/2}}{\alpha^{1/p}}\right)^{\mu/\lambda}.
\]
Similar arguments are applicable for $p=\infty$, in which case ${1}/{p}$ should be replaced with $0$.

\end{document}